\documentclass[a4paper,10pt,reqno, english]{amsart}  
\usepackage[utf8]{inputenc}
\usepackage[T1]{fontenc}
\usepackage{amsmath,amsthm}
\usepackage{amsfonts,amssymb,enumerate}
\usepackage{url,paralist}
\usepackage{mathtools}  
\usepackage[colorlinks=true,urlcolor=blue,linkcolor=red,citecolor=magenta]{hyperref}
\usepackage{enumerate}
\usepackage[left=1in,right=1in,top=1in,bottom=1in]{geometry}

\theoremstyle{plain}
\newtheorem{theorem}{Theorem}[section]
\newtheorem{lemma}[theorem]{Lemma}

\newtheorem{corollary}[theorem]{Corollary}

\newtheorem*{theorem*}{Theorem}

\newtheorem*{claim*}{Claim}

\theoremstyle{definition}

\newtheorem{example}[theorem]{Example}
\newtheorem{remark}[theorem]{Remark}

\newcommand{\R}{\mathbb{R}}
\newcommand{\Z}{\mathbb{Z}}

\begin{document}

\title[Borsuk--Ulam theorems]{Borsuk--Ulam
theorems for products of spheres \\ and Stiefel manifolds revisited}



\author{Yu Hin Chan} 
\address[YC]{Dept.\ Math., University of California at Davis, Davis, CA}
\email{yuhchan@math.ucdavis.edu} 

\author{Shujian Chen}
\address[SC]{Dept.\ Math., Brandeis University, Waltham, MA}
\email{shujianchen@brandeis.edu} 

\author{Florian Frick}
\address[FF]{Dept.\ Math.\ Sciences, Carnegie Mellon University, Pittsburgh, PA}
\email{frick@cmu.edu} 

\author{J. Tristan Hull}
\address[JTH]{Dept.\ Math., University of California at Berkeley, Berkeley, CA}
\email{jth242@berkeley.edu}


\begin{abstract}
\small
We give a different and possibly more accessible proof of a general Borsuk--Ulam theorem for a product of spheres, originally due to Ramos. That is, we show the non-existence of certain $(\Z/2)^k$-equivariant maps from a product of $k$ spheres to the unit sphere in a real $(\Z/2)^k$-representation of the same dimension. Our proof method allows us to derive Borsuk–Ulam theorems for certain equivariant maps from Stiefel manifolds, from the corresponding results about products of spheres, leading to alternative proofs and extensions of some results of Fadell and Husseini.
\end{abstract}

\date{February 13, 2019, revised August 12, 2019}
\maketitle

\section{Introduction}

Let $X$ be a compact $n$-dimensional CW complex with an action by the group~$G$. A fundamental question with a multitude of applications in topological combinatorics is to decide whether an equivariant map $X \longrightarrow V$ (that is, a map commuting with a $G$-action) into some $n$-dimensional real $G$-representation $V$ must have $0 \in V$ in its image. Equivalently, one is interested in deciding the existence of an equivariant map $X \longrightarrow S(V)$ into the unit sphere of~$V$. This method has found applications in hyperplane mass partitions~\cite{blagojevic2018}, the ``square-peg'' problem~\cite{matschke2014}, Tverberg-type results~\cite{blagojevic2017}, and chromatic numbers of hypergraphs~\cite{matousek2002}, among others; see~\cite{matousek2008, zivaljevic2017}. Thus the identification of easily computable obstructions to the existence of such equivariant maps is of fundamental importance.

One incarnation of this problem that has received particular attention is the case that $X$ is a product of spheres $S^{n_1} \times \dots \times S^{n_k}$, and $G$ is $(\Z/2)^k$ where the $j$th copy of~$\Z/2$, generated by~$\varepsilon_j$, acts non-trivially exactly on the $j$th factor~$S^{n_j}$. The case $k=1$ is the classical Borsuk--Ulam theorem, which states that there is no map $S^n \longrightarrow S^{n-1}$ that commutes with the antipodal actions. Extensions of this result to products of spheres have been studied also because such maps naturally appear for the problem of equipartitions by hyperplanes; see for example~\cite{mani2006, blagojevic2018, simon2019}.

Here we study a binary-valued obstruction for the existence of equivariant maps: The parity of the sum of degrees of a specially extended map restricted to various submanifolds obstructs the existence of an equivariant map. This yields a possibly more accessible proof of Ramos' general Borsuk--Ulam theorem for products of spheres~\cite{ramos1996}. Moreover, our reasoning extends to Stiefel manifolds~$V_{n,k}$ of $k$ mutually orthonormal vectors in~$\R^n$. A classical result of Fadell and Husseini~\cite{fadell1988} establishes the non-existence of an equivariant map from $V_{n,k}$ into~$S(V)$, where $V$ is $(\R^{n-k})^k$ and $\varepsilon_j$ acts non-trivially on the $j$th copy of~$\R^{n-k}$. The difference in dimensions of $V_{n,k}$ and $(\R^{n-k})^k$ is~$\binom{k}{2}$, which leaves room for improvement of Fadell and Husseini's result. It seems that the following theorem has not been recorded before, although (as pointed by an anonymous referee) it is implicit in~\cite{fadell1988}, where the difference of dimensions is also exploited in a different way:

\begin{theorem}
\label{thm:main}
	Every $(\Z/2)^k$-equivariant map 
	$$V_{n,k} \longrightarrow \R^{n-1} \oplus \R^{n-2} \oplus \dots \oplus \R^{n-k}$$
	has a zero. Here $\varepsilon_j$ acts non-trivially precisely on the $j$th factor~$\R^{n-j}$ and by $(x_1, \dots, x_j, \dots, x_n) \mapsto (x_1, \dots, -x_j, \dots, x_n)$ on~$V_{n,k}$.
\end{theorem}

We prove a more general result for arbitrary $(\Z/2)^k$-actions on the codomain; see Theorem~\ref{thm:main2}. We also show that the zeros of an equivariant map $V_{n,k} \longrightarrow (\R^{n-k})^k$ can be restricted to lie in a fixed submanifold of codimension~$\binom{k}{2}$; see Corollary~\ref{cor:main}.
Ramos' approach is elementary and technical, depending on equivariant approximations of sufficiently generic PL maps for appropriately defined triangulations. Some effort has been invested into simplifying his proofs. For example, a special case was proved by Dzedzej, Idzik, and Izydorek~\cite{dzedzej1999}. Further, we mention that some results of Ramos have been reproven since gaps have been pointed out in the treatment of non-free actions~\cite{blagojevic2018} (but not for the treatment of free actions as in the present manuscript). These  results have been salvaged by different methods~\cite{blagojevic2016}, and Vre\'cica and \v Zivaljevi\'c have proposed a supplement for Ramos' proof~\cite{vrecica2015}. Several authors have studied the existence of $(\Z/2)^k$-equivariant maps from products of spheres to a sphere using the theory of Fadell and Husseini, such as Mani-Levitska, Vre\' cica, and \v Zivaljevi\'c~\cite{mani2006}, Blagojevi\'c and Ziegler~\cite{blagojevic2011}, and Simon~\cite{simon2019}. For computations of Fadell and Husseini's cohomological index of Stiefel manifolds see Inoue~\cite{inoue2006} and Blagojevi\'c and Karasev~\cite{blagojevic2012}.

\section{Borsuk--Ulam theorems for products of spheres}

We denote the standard generators of $(\Z/2)^k$ by $\varepsilon_1, \dots, \varepsilon_k$. We think of $\Z/2 = \{0,1\}$ additively and write $\langle \alpha, \beta \rangle = \sum_j \alpha_j\beta_j \in \Z/2$ for the inner product in~$(\Z/2)^k$. For $\alpha \in (\Z/2)^k$ denote by $V_\alpha$ the vector space $\R$ with the action of $(\Z/2)^k$ where $\varepsilon_j$ acts non-trivially by $x \mapsto -x$ if $\langle \varepsilon_j, \alpha_i \rangle = 1$ and trivially otherwise. Denote the closed upper hemisphere of $S^{n_i}$ by~$B^{n_i}$ and let $B = B^{n_1} \times \dots \times B^{n_k}$. If the equivariant map $f\colon S^{n_1} \times \dots \times S^{n_k} \longrightarrow V$ is never zero on~$\partial B$, then $f$ induces a map $\widehat f \colon \partial B \longrightarrow S(V), x \mapsto f(x) / |f(x)|$. We will show that the parity of the degree of this map is independent of $f$ and only depends on the module~$V$ and the numbers $n_1, \dots, n_k$. We denote the degree of $\widehat f$ modulo $2$ by $r(n_1, \dots, n_k; V) \in \Z/2$. 

The degree of $\widehat f$ equivalently counts the number of zeros of $f$ in $B$ counted with signs and multiplicities. The notion of sign and multipicity here is captured by the local degree: Let $X$ and $Y$ be oriented closed $n$-dimensional manifolds, $x \in X$, and $f\colon X \longrightarrow Y$ a continuous map. Then $f$ induces a map $f_*\colon H_{n}(X, X\setminus \{x\}) \longrightarrow H_{n}(Y, Y \setminus \{f(x)\})$. Both the domain and codomain of this homomorphism are isomorphic to~$\Z$, and thus $f_*$ is uniquely determined by $d = f_*(1)$, the local degree $\deg f|_x$ of $f$ around~$x$. We refer to Outerelo and Ruiz for the basics of mapping degree theory~\cite{outerelo2009}. For example, they prove (see \cite[Prop.~4.5]{outerelo2009}):

\begin{lemma}
\label{lem:compute-deg}
	Let $W$ be a compact, oriented $(n+1)$-manifold with boundary~$X$. Let $f\colon W \longrightarrow \R^{n+1}$ be continuous with $f^{-1}(0)$ finite and disjoint from~$X$. Then the degree of the map $\widehat f\colon X \longrightarrow S^n, x \mapsto f(x) / |f(x)|$ is the sum of local degrees of $f$ around its zeros:
	$$\deg \widehat f = \sum_{x\in f^{-1}(0)} \deg f|_x.$$
\end{lemma}

While in~\cite{outerelo2009} this lemma is stated in the smooth category for maps with regular value~$0$, it is simple to see that the lemma holds in this slightly more general setting. See for instance~{\cite[Proof of Lemma~5.6]{blagojevic2018}} for a proof.

As a second ingredient we need that if the action on the domain is free, any two $G$-equivariant maps have congruent degrees modulo the order of~$G$; see Kushkuley and Balanov~{\cite[Cor.~2.4]{kushkuley2006}}:

\begin{theorem}
\label{thm:congruent-deg}
    Let $X$ and $Y$ be closed oriented $n$-dimensional manifolds with actions by the finite group~$G$, such that the $G$-action on $X$ is free. Then for any two equivariant maps $f_1, f_2\colon X \longrightarrow Y$ their degrees are congruent modulo~$|G|$: $$\deg f_1 \equiv \deg f_2 \mod |G|.$$
\end{theorem}

For $x \in S^n \subset \R^{n+1}$ we denote the $i$th coordinate of~$x$ by~$e_i^*(x)$. We fix the upper hemisphere $B^n \subset S^n$ as the set of $x \in S^n$ with $e_{n+1}^*(x) \ge 0$. We can now prove that $r(n_1, \dots, n_k; V)$ is indeed independent of~$f$:

\begin{lemma}
	Let $f \colon S^{n_1} \times \dots \times S^{n_k} \longrightarrow V$ be a $(\Z/2)^k$-equivariant map that is never zero on~$\partial B$ and has only isolated zeros. Then the degree modulo $2$ of the induced map $\widehat f \colon \partial B \longrightarrow S(V)$ is independent of~$f$.
\end{lemma}

\begin{proof}
	The degree of $\widehat f$ is equal to the sum of zeros of $f$ in $B$ counted with sign and multiplicity by Lemma~\ref{lem:compute-deg}. Denote by $M = S^{n_1-1} \times S^{n_2} \times \dots \times S^{n_k}$. The degree modulo $2^k = |(\Z/2)^k|$ of the map $f|_M$ is independent of~$f$ by Theorem~\ref{thm:congruent-deg}. This degree counts the number of zeros of $f$ in $W = B^{n_1} \times S^{n_2} \times \dots \times S^{n_k}$ with signs and multiplicities, again by Lemma~\ref{lem:compute-deg}. Every zero of $f$ in $B$ occurs $2^{k-1}$ times in~$W$ by symmetry. If these $2^{k-1}$ symmetric copies of a zero in $B$ all have the same sign---this is the case if each $\varepsilon_j$, $j \ge 2$, preserves orientation on $M$ if and only if it preserves the orientation of~$S(V)$---then $\deg f|_M$ is the sum of zeros of $f$ in $B$ multiplied by~$2^{k-1}$. Since $\deg f|_M$ modulo $2^k$ is independent of~$f$, so is the parity of $\deg \widehat f$. 
	
	We now induct on the number of generators $\varepsilon_j$ that act in opposite ways on the orientation of $M$ and~$S(V)$. Suppose we have already shown that no matter how the generators $\varepsilon_j$, $2 \le j \le \ell -1$ act on the orientations of $M$ and~$S(V)$, the parity of the degree $\deg \widehat f$ is independent of~$f$. Further, assume that $\varepsilon_\ell$ acts orientation-preservingly on~$M$ and orientation-reversingly on~$S(V)$ or vice versa. Let $M' = S^{n_1+1} \times S^{n_2} \times \dots \times S^{n_k}$ and $B' = B^{n_1+1} \times \dots \times B^{n_k}$. Extend $f$ equivariantly to a map $f'\colon M' \longrightarrow V$. In this extension process ensure that $f'$ has only finitely many zeros $x = (x_1, \dots, x_k)$ with $e_1^*(x_1)e_1^*(x_\ell) = 0$. This is possible since $f$ has no zeros with $e_1^*(x_1) =0 $ or $e_1^*(x_\ell) = 0$, and any map $S^{n-1} \longrightarrow S^{n-1}$ can be extended to the entire ball $B^n \longrightarrow \R^n$ such that it only has finitely many zeros. Now consider the map 
	$$F\colon M' \longrightarrow V \oplus V_\alpha, \qquad (x_1, \dots, x_k) \mapsto (f'(x_1, \dots, x_k), e_1^*(x_1)e_1^*(x_\ell)),$$
	where $\alpha = \varepsilon_1 + \varepsilon_\ell$. The map $F$ is equivariant and has only isolated zeros. Further, by perhaps slightly rotating some of the spheres in the domain $M'$, we can guarantee that $F$ has no zeros in~$\partial B'$, since $F$ has only finitely many zeros. 
	
	Now $\varepsilon_\ell$ acts in the same way on the orientations of $M'$ and $S(V \oplus V_\alpha)$, since $\varepsilon_\ell$ acts in opposite ways on $S(V)$ and $S(V\oplus V_\alpha)$. Thus by induction the parity of the degree of the induced map $\widehat F\colon \partial B' \longrightarrow S(V \oplus V_\alpha)$ does not depend on~$F$. This counts the zeros of $F$ in~$B'$ with signs and multiplicities. These zeros are precisely the zeros of $f$ in $B$ and the zeros of $F$ in $B^{n_1+1} \times B^{n_2} \times \dots \times B^{n_\ell -1} \times \dots \times B^{n_k}$. But the parity of the latter number of zeros does not depend on $F$ by induction. Thus the parity of the number of zeros of $f$ in $B$ does not depend on $f$ either.
\end{proof}

That the parity of $\deg \widehat f$ is independent of~$f$ could also be derived as a consequence of elementary obstruction theory. We briefly sketch this argument and refer to tom Dieck~\cite{tomdieck2011} for the basics of (equivariant) obstruction theory. Let $X$ be an $n$-dimensional CW complex with a free cellular $G$-action. Denote the $k$-skeleton of~$X$ by~$X^{(k)}$. Let $Y$ be an $(n-2)$-connected, $(n-1)$-simple $G$-space. Then there is a $G$-map $h \colon X^{(n-1)} \longrightarrow Y$. Whether $h$ can be extended (up to homotopy on~$X^{(n-2)}$) to a $G$-map defined on all of~$X$ is captured by the obstruction cocycle $\mathfrak o \in H^n_G(X; \pi_{n-1} Y)$. In the situation described here---a primary obstruction problem---the cohomology class $\mathfrak o$ is independent of the map $h$, and a $G$-map $X \longrightarrow Y$ exists if and only if the cohomology class $\mathfrak o$ vanishes.

If $Y = S^{n-1}$ then the value (of a representative) of $\mathfrak o$ on an $n$-cell $\sigma$ of $X$ is the degree of $h$ restricted to~$\partial \sigma$. A sphere has a $\Z/2$-equivariant CW structure with two cells in each dimension. This induces a CW complex structure on a product of spheres that is equivariant with respect to~$(\Z/2)^k$. This CW structure has one orbit of top-dimensional $n$-cells, and thus $\deg \widehat f$ determines~$\mathfrak o$. Each orbit of $(n-1)$-cells intersects the boundary of a fixed $n$-cell in an even number of cells. Since each equivariant $(n-1)$-cochain has the same value up to signs on the cells of the same orbit, the parity of the value of $\mathfrak o$ on each $n$-cell is well-defined (and thus independent of the map~$h$).

\begin{lemma}
\label{lem:zeros}
	Let $f\colon S^{n_1} \times \dots \times S^{n_k} \longrightarrow V$ be an equivariant map that is never zero on $\partial B$, has only isolated zeros, and is a local homeomorphism around zeros. Then $r(n_1,\dots,n_k;V) = 1$ if and only if $f$ has an odd number of zeros in~$B$. Moreover, if an equivariant map $S^{n_1} \times \dots \times S^{n_k} \longrightarrow S(V)$ exists, then $r(n_1,\dots,n_k;V) = 0$.
\end{lemma}

\begin{proof}
	Since $f$ is a local homeomorphism around zeros, every local degree $\deg f|_x$ for $x$ with $f(x) = 0$ is $\pm 1$. By Lemma~\ref{lem:compute-deg} the sum $\sum_{x\in f^{-1}(0) \cap B} \deg f|_x$ is odd if and only if $r(n_1,\dots,n_k;V) = 1$. This is the case precisely if $f^{-1}(0) \cap B$ has an odd number of elements. The second statement is an immediate consequence of Lemma~\ref{lem:compute-deg}.
\end{proof}

It is now an elementary exercise to prove the following result, originally due to Ramos~\cite{ramos1996}. 

\begin{theorem}[Ramos~{\cite[Section 3]{ramos1996}}]
\label{thm:compute}
	The value of $r(n_1, \dots, n_k; V)$ can be computed recursively via
	\begin{align}
	 r(n_1, \dots,n_k; \bigoplus_{i=1}^n V_{\alpha_i}) = \sum_{j=1}^k \langle \alpha_n, \varepsilon_j \rangle r(n_1, \dots, n_j -1, \dots, n_k;  \bigoplus_{i=1}^{n-1} V_{\alpha_i}).
	\end{align}
\end{theorem}

\begin{proof} 
	Let $f\colon S^{n_1} \times \dots \times S^{n_k} \longrightarrow \bigoplus_{i=1}^{n-1} V_{\alpha_i}$ be an equivariant map such that for $(x_1, \dots, x_k) \in S^{n_1} \times \dots \times S^{n_k}$ with $e_1^*(x_j)e_{n_j+1}^*(x_j) = 0 = e_1^*(x_\ell)$ for $j \ne \ell$, $f(x_1, \dots, x_k) \ne 0$. Such a map exists by obstruction theory since $S(\bigoplus_{i=1}^{n-1} V_{\alpha_i})$ is $(n-3)$-connected. In the next step of the construction of~$f$, where we define $f$ on the $(n-1)$-skeleton, we can moreover ensure that $f$ has finitely many zeros $x$ with $e_1^*(x_j) = 0$ for some~$j$. This is because cell-by-cell a map $S^{n-2} \longrightarrow S^{n-2}$ can be extended to a map $B^{n-1} \longrightarrow \R^{n-1}$ with finitely many zeros. Now define
	$$F\colon S^{n_1} \times \dots \times S^{n_k} \longrightarrow \bigoplus_{i=1}^n V_{\alpha_i}, \qquad (x_1, \dots, x_k) \mapsto (f(x_1, \dots, x_k), \prod_{j : \langle \alpha_n, \varepsilon_j \rangle = 1} e_1^*(x_j)).$$
	The product in the last coordinate ensures that $F$ is also equivariant in the $V_{\alpha_n}$-component. Now observe that the zeros of $F$ in $B^{n_1} \times \dots \times B^{n_k}$ are in bijection with zeros of $f$ in $B^{n_1} \times \dots \times B^{n_k}$ that satisfy~${e_1^*(x_j) = 0}$, where $j$ ranges over indices with $\langle \alpha_n, \varepsilon_j \rangle = 1$. Using Lemma~\ref{lem:zeros} and reducing modulo~$2$ finishes the proof.
\end{proof}

\begin{remark}
\label{rem:boundary-cases}
The value of $r(n_1, \dots, n_k; V)$ is well-defined even in the case that some $n_i$ are~$0$. In this case the manifold $X = S^{n_1} \times \dots \times S^{n_k}$ splits into several connected components and thus $H_{n-1}(\partial B) \cong \Z^c$, where $c$ is the number of components, that is, $c = 2^t$ with $t$ the number of $S^0$-factors in~$X$. In this case we define the degree of the map $\widehat f \colon \partial B \longrightarrow S(V)$ to be $\widehat f_*(1,\dots, 1)$. Since this counts the parity of the number of zeros of $f$ in $B$ with signs and multiplicities, our results hold in the same way, even if some $n_i = 0$. Further, an equivariant map $S^{n_1} \times \dots \times S^{n_{k-1}} \times S^0 \longrightarrow S(V)$ exists if and only if an equivariant map $S^{n_1} \times \dots \times S^{n_{k-1}} \longrightarrow S(V)$ exists, where we forget the action of $\varepsilon_k$ to define the induced $(\Z/2)^{k-1}$-actions. This is true simply because $S^{n_1} \times \dots \times S^{n_{k-1}} \times S^0$ consists of two disjoint copies of $S^{n_1} \times \dots \times S^{n_{k-1}}$. For the nontrivial $\Z/2$-action on $\R$ we have that $r(1; \R) = 1$; this is the Intermediate Value Theorem. Using these observations and Theorem~\ref{thm:compute}, we can compute the value of the obstruction $r(n_1, \dots, n_k; V)$ by induction.
\end{remark}

Other Borsuk--Ulam results for products of spheres are corollaries of Theorem~\ref{thm:compute}: 

\begin{remark}
An immediate consequence of Theorem~\ref{thm:compute} is 
\begin{align}
\label{eq}
 r(n_1, \dots, n_k; W \oplus V_{\varepsilon_j}) = r(n_1, \dots, n_j-1, \dots, n_k; W)
\end{align}
and thus $r(n_1, \dots, n_k; V_{\varepsilon_1}^{\oplus n_1} \oplus \dots \oplus V_{\varepsilon_k}^{\oplus n_k}) = 1$. This implies that any $(\Z/2)^k$-map $S^{n_1} \times \dots \times S^{n_k} \longrightarrow V_{\varepsilon_1}^{\oplus n_1} \oplus \dots \oplus V_{\varepsilon_k}^{\oplus n_k}$ has a zero---a proof of this special case of Ramos' result using the cohomological index theory of Fadell and Husseini is due to Dzedzej, Idzik, and Izydorek~\cite{dzedzej1999}. 
\end{remark}

\begin{remark}
\label{rem:common1}
As another special case we remark that 
\begin{align}
 r(n_1,n_2; W\oplus V_{\varepsilon_1+\varepsilon_2}) = r(n_1-1,n_2; W) + r(n_1,n_2-1; W),
\end{align}
and thus $r(n_1,n_2; W\oplus V_{\varepsilon_1+\varepsilon_2}) = 1$ if and only if $r(n_1-1,n_2; W) \ne r(n_1,n_2-1; W)$. This is the same recursion that computes binomial coefficients $\binom{n_1+n_2}{n_1}$, and thus $r(n_1,n_2; V_{\varepsilon_1+\varepsilon_2}^{\oplus (n_1 + n_2)}) = 1$ if and only if $\binom{n_1+n_2}{n_1}$ is odd. This is the case if and only if the binary expansions of $n_1$ and $n_2$ do not share a common~$1$.
\end{remark}

\begin{remark}
Combining the observations of the previous two remarks shows that
$$r(3\cdot 2^t-1, 3\cdot 2^t-2 ; (V_{\varepsilon_1} \oplus V_{\varepsilon_2} \oplus V_{\varepsilon_1+\varepsilon_2})^{\oplus (2^{t+1}-1)}) = r(2^t, 2^t-1 ; V_{\varepsilon_1+\varepsilon_2}^{\oplus (2^{t+1}-1)}) = 1,$$
obstructing the existence of a $(\Z/2)^2$-equivariant map $S^{3n-1} \times S^{3n-2} \longrightarrow S((V_{\varepsilon_1} \oplus V_{\varepsilon_2} \oplus V_{\varepsilon_1+\varepsilon_2})^{\oplus (2n-1)})$ for $n = 2^t$. This is a result of Mani-Levitska, Vre\'cica, and \v Zivaljevi\'c~\cite{mani2006}.
\end{remark}

\section{Borsuk--Ulam theorems for Stiefel manifolds}

The purpose of this section is to derive our main result, Theorem~\ref{thm:main}, from the methods developed in the preceding section. Recall that for positive integers $k \le n$, $V_{n,k}$ denotes the Stiefel manifold of $k$ pairwise orthonormal vectors in~$\R^n$. In particular, $V_{n,k} \subset (S^{n-1})^k$ and $V_{n,k}$ is invariant under the action of $(\Z/2)^k$ and thus inherits an action of~$(\Z/2)^k$. We will strengthen the following result of Fadell and Husseini:

\begin{theorem}[Fadell and Husseini~\cite{fadell1988}]
\label{thm:fh}
	Every equivariant map 
	$$V_{n,k} \longrightarrow (V_{\varepsilon_1} \oplus \dots \oplus V_{\varepsilon_k})^{\oplus (n-k)}$$
	has a zero.
\end{theorem}

The dimension of $V_{n,k}$ exceeds the dimension of the codomain $(V_{\varepsilon_1} \oplus \dots \oplus V_{\varepsilon_k})^{\oplus (n-k)}$ by~$\binom{k}{2}$. We give two strengthenings of Theorem~\ref{thm:fh}, where the dimensions of domain and codomain coincide. This is possible by restricting the domain to an invariant submanifold of codimension~$\binom{k}{2}$ (this is achieved by Corollary~\ref{cor:main}), or by mapping to a larger codomain instead while still guaranteeing the existence of a zero---this is Theorem~\ref{thm:main}, which is already implicit in~\cite{fadell1988}. Moreover, we prove a generalized result that applies to arbitrary $(\Z/2)^k$-actions on the codomain; see Theorem~\ref{thm:main2}.

For $\alpha \in (\Z/2)^k$ denote by $|\alpha|$ the $\ell_1$-norm of~$\alpha$, that is, the number of non-zero entries. The dimension of $\bigoplus_{|\alpha| = 2} V_\alpha$ is $\binom{k}{2} = \sum_{i=0}^{k-1} i$. We will need the following:

\begin{lemma}
\label{lem:reduction}
	$r(k-1,k-2,\dots,1,0; \bigoplus_{|\alpha| = 2} V_\alpha) = 1$
\end{lemma}

We provide two short proofs of this lemma. The first proof using the recursive formula of Theorem~\ref{thm:compute}, the second by exhibiting an appropriate equivariant map with an odd number of zeros in a fundamental domain of the $(\Z/2)^k$-action. 

\begin{proof}[Proof 1 of Lemma~\ref{lem:reduction}]
	Let $W = \bigoplus_{\alpha \in (\Z/2)^{k-1}, |\alpha| = 2} V_\alpha$ and $U = \bigoplus_{j=1}^{k-1} V_{\varepsilon_j}$. By Remark~\ref{rem:boundary-cases} and using Equation~(\ref{eq}) above $k-1$ times,
	$$r(k-1,k-2,\dots,1,0; \bigoplus_{|\alpha| = 2} V_\alpha) = r(k-1,k-2,\dots,1; W \oplus U) = r(k-2,k-3,\dots,1,0; W).$$
	Thus the lemma follows by induction on~$k$. The base case is the Intermediate Value Theorem.
\end{proof}

\begin{proof}[Proof 2 of Lemma~\ref{lem:reduction}]
	Consider a filtration $S^0 \subset S^1 \subset \dots \subset S^{k-1}$ obtained by successively intersecting $S^{k-1} \subset \R^k$ by coordinate hyperplanes. Then it is simple to check that the map
	$$S^{k-1} \times \dots \times S^1 \times S^0 \longrightarrow \bigoplus_{|\alpha| = 2} V_\alpha, \qquad (x_1, \dots, x_k) \mapsto (\langle x_i, x_j \rangle)_{i < j}$$
	is equivariant and has exactly one zero up to symmetry.
\end{proof}

We can adapt the reasoning above to show a Borsuk--Ulam type theorem for Stiefel manifolds. The proof of the following theorem uses the same reasoning as~{\cite[Prop.~3.3]{blagojevic2018}}, which was previously used to derive Tverberg-type results in~\cite{blagojevic2014}.

\begin{theorem}
\label{thm:main2}
	Let $k \le n$ be integers and $m = k(n-1) - \binom{k}{2}$. Let $\alpha_1, \dots, \alpha_m \in (\Z/2)^k$, and denote $\bigoplus_{i=1}^m V_{\alpha_i}$ by~$V$. If $r(\underbrace{n-1, \dots, n-1}_\text{$k$ times}; V \oplus \bigoplus_{|\alpha| = 2} V_\alpha) = 1$ then there is no equivariant map $V_{n,k} \longrightarrow S(V)$.
\end{theorem}

\begin{proof}
	We prove the contrapositive. Given an equivariant map $f\colon V_{n,k} \longrightarrow S(V)$, extend it to an equivariant map $f'\colon (S^{n-1})^k \longrightarrow V$.
	We could argue that this extension exists by exhibiting a cell complex for $(S^{n-1})^k$ that respects the $(\Z/2)^k$-action and has $V_{n,k}$ as a subcomplex. Then by contractibility of $V$ and the action of $(\Z/2)^k$ being free, there is no obstruction to the existence of this extension. Here we instead give an explicit formula for~$f'$: Let $(x_1, \dots, x_k) \in (S^{n-1})^k$. To define $f'(x_1,\dots,x_k)$, first inductively define points $y_1, \dots, y_k \in \R^n$. Think of $S^{n-1}$ as the unit sphere in~$\R^n$. Set $y_1 = x_1$, and having defined $y_1, \dots, y_{j-1}$, let $y_j = \pi_j(x_j)$, where $\pi_j \colon \R^n \longrightarrow \R^n$ is the orthogonal projection onto the orthogonal complement of the subspace spanned by $x_1, \dots, x_{j-1}$. Further, whenever $y_j \ne 0$, let $y_j' = \frac{y_j}{|y_j|}$. Now if all $y_j$ are non-zero, define $f'(x_1, \dots, x_k)$ by $f(y_1',\dots, y_k')\cdot \prod |y_j|$, whereas if some $y_j = 0$, then let $f'(x_1,\dots,x_k) = 0$. This map is continuous: The map $f$ as a continuous map on the compact manifold $V_{n,k}$ is bounded and so are the $|y_j|$. As $(x_1,\dots,x_k)$ approaches a point where $x_j$ is in the subspace spanned by $x_1, \dots, x_{j-1}$ and thus $y_j$ approaches zero, the value of $f'(x_1, \dots, x_k)$ approaches zero as well. Moreover, $f'$ is equivariant: flipping $x_j$ to $-x_j$ leaves $y_\ell$ for $\ell \ne j$ unchanged and flips $y_j'$ to $-y_j'$ (so long as $y_j \ne 0$). The equivariance of $f'$ now follows from the equivariance of~$f$.
	
	Now define the equivariant map $$F\colon  (S^{n-1})^k \longrightarrow V \oplus  \bigoplus_{|\alpha| = 2} V_\alpha, \qquad (x_1, \dots, x_k) \mapsto f'(x_1,\dots,x_k) \oplus (\langle x_i, x_j \rangle)_{i < j}.$$
	The map $F$ does not have any zeros, since $F(x_1, \dots, x_k) = 0$ implies that $f'(x_1,\dots,x_k) = 0$ and the $x_i$ are mutually orthogonal, thus $(x_1,\dots,x_k) \in V_{n,k}$ is a zero of~$f$. As $F$ does not have any zeros, $r(n-1, \dots, n-1; V \oplus \bigoplus_{|\alpha| = 2} V_\alpha) = 0$.
\end{proof}

\begin{example}
	We provide an example of a consequence of Theorem~\ref{thm:main2}. Let $k = 2$ and let $n-1$ be a power of two. By Remark~\ref{rem:common1} we have $r(n-1,n-2; (V_{\varepsilon_1 + \varepsilon_2})^{\oplus(2n-3)}) = 1$ and thus, in particular, there is no $(\Z/2)^2$-map $V_{n,2} \longrightarrow S(V_{\varepsilon_1 + \varepsilon_2}^{\oplus(2n-4)})$.
\end{example}

We can now derive Theorem~\ref{thm:main} that any equivariant map $$V_{n,k} \longrightarrow V_{\varepsilon_1}^{\oplus(n-1)} \oplus V_{\varepsilon_2}^{\oplus(n-2)} \oplus \dots \oplus V_{\varepsilon_k}^{\oplus (n-k)}$$ has a zero.

\begin{proof}[Proof of Theorem~\ref{thm:main}]
	According to Theorem~\ref{thm:main2}, we need to show that $$r(\underbrace{n-1, \dots, n-1}_\text{$k$ times}; V \oplus \bigoplus_{|\alpha| = 2} V_\alpha) = 1,$$ where $V = V_{\varepsilon_1}^{\oplus(n-1)} \oplus V_{\varepsilon_2}^{\oplus(n-2)} \dots \oplus V_{\varepsilon_k}^{\oplus (n-k)}$. By Theorem~\ref{thm:compute}
	$$r(n-1, \dots, n-1; V \oplus \bigoplus_{|\alpha| = 2} V_\alpha) = r(k-1, k-2, \dots, 1,0; \bigoplus_{|\alpha| = 2} V_\alpha),$$
	which is equal to $1$ by Lemma~\ref{lem:reduction}.
\end{proof}

Lastly, we can strengthen Theorem~\ref{thm:fh} by showing that there always is a zero of an equivariant map from the Stiefel manifold~$V_{n,k}$ that lies in some proper fixed submanifold. Let $M = \{(x_1,\dots, x_k) \in S^{n-k} \times S^{n-k+1} \times \dots \times S^{n-1} \: : \: \langle x_i, x_j \rangle = 0 \quad \forall i \ne j\}$. In particular, $M$ is a $(\Z/2)^k$-invariant submanifold of $V_{n,k}$ of codimension~$\binom{k}{2}$.

\begin{corollary}
\label{cor:main}
	Any equivariant map $M \longrightarrow (V_{\varepsilon_1} \oplus \dots \oplus V_{\varepsilon_k})^{\oplus (n-k)}$ has a zero.
\end{corollary}

\begin{proof}
	Given any such equivariant map $f\colon M \longrightarrow (V_{\varepsilon_1} \oplus \dots \oplus V_{\varepsilon_k})^{\oplus (n-k)}$ extend it to an equivariant map $f'\colon V_{n,k} \longrightarrow (V_{\varepsilon_1} \oplus \dots \oplus V_{\varepsilon_k})^{\oplus (n-k)}$. This extension exists for the same reason as in the proof of Theorem~\ref{thm:main2}, but now the explicit construction is simpler: Think of $S^{n-k} \subset S^{n-k+1} \subset \dots \subset S^{n-1} \subset \R^n$ as a filtration of the unit sphere in $\R^n$ by successive equatorial spheres. Let $(y_1, \dots, y_k) \in V_{n,k}$ such that $y_j$ is orthogonal to all points in~$S^{n-k+j}$. Declare $f'(y_1, \dots, y_k)$ to be zero and extend linearly along geodesics to~$S^{n-k+j}$.
	
	There is an equivariant map
	$$h\colon V_{n,k} \longrightarrow V_{\varepsilon_1}^{\oplus(k-1)} \oplus V_{\varepsilon_2}^{\oplus(k-2)} \oplus \dots \oplus V_{\varepsilon_{k-2}}^{\oplus 2} \oplus V_{\varepsilon_{k-1}}$$ with $h^{-1}(0) = M$: Explicitly, the $V_{\varepsilon_\ell}^{\oplus(k-\ell)}$-component, $\ell < k$, of $h(x_1, \dots, x_k)$ is given by $$(e_n^*(x_\ell), e_{n-1}^*(x_\ell), \dots, e_{n-(k-\ell)+1}^*(x_\ell)).$$
	Then define the equivariant map
	$$F\colon  V_{n,k} \longrightarrow (V_{\varepsilon_1} \oplus \dots \oplus V_{\varepsilon_k})^{\oplus (n-k)} \oplus V_{\varepsilon_1}^{\oplus(k-1)} \oplus V_{\varepsilon_2}^{\oplus(k-2)} \oplus \dots \oplus V_{\varepsilon_{k-2}}^{\oplus 2} \oplus V_{\varepsilon_{k-1}}, x \mapsto (f'(x), h(x)).$$
	By Theorem~\ref{thm:main} the map $F$ has a zero. Now, as before, if $F(x) = 0$, then both $f'(x) = 0$ and $h(x) = 0$. The latter implies $x \in M$, while $f'(x) = 0$ means that $f(x) = 0$. 
\end{proof}

\section*{Acknowledgements}

We thank the two referees for their thorough reading of our manuscript and numerous thoughtful comments, which improved the exposition substantially.
These results were obtained during the \emph{Summer Program for Undergraduate Research} 2017 at Cornell University. The authors are grateful for the excellent research conditions provided by the program. The authors would like to thank Maru Sarazola for many insightful conversations and Pavle Blagojevi\'c for pointing out an additional relevant reference.

\bibliographystyle{amsplain}


\end{document}